\documentclass[a4paper,11pt]{article}

\usepackage{amsmath,amsfonts,amssymb,mathtools,amsthm}
\usepackage{bbm} 
\usepackage{epsfig,mathrsfs,enumerate,graphics}
\usepackage[usenames,dvipsnames]{color}
\usepackage{dsfont}
\usepackage{graphicx}
\usepackage{adjustbox}
\usepackage{upgreek}
\usepackage{wrapfig}
\usepackage{caption}
\usepackage{subcaption}
\usepackage{tikz}

\usepackage[all]{xy} \entrymodifiers={!!<0pt,0.7ex>+}

\usepackage[hyperindex=true, 
					hidelinks]{hyperref}

\newcommand{\bbl}{\color{black}}

\newcommand{\ens}{\enspace}

\newtheorem{thm}{Theorem}

\newtheorem{lem}{Lemma}[section]

\newtheorem{cor}[lem]{Corollary}

\theoremstyle{definition}

\newtheorem{Def}[lem]{Definition}

\numberwithin{equation}{section}

\newcounter{parag}[subsection]

\newcounter{parage}[section]
\renewcommand{\theparage}{{\bf\thesection.\arabic{parage}}}
\newcommand{\parage}{\medskip \addtocounter{parage}{1} 
\noindent{\theparage\ } }

\newcounter{paraga}

\def\al{\alpha}
\def\be{\beta}
\def\ga{\gamma}

\def\de{\delta}

\def\De{\Delta}
\def\eps{{\varepsilon}}

\def\om{\omega}

\def\Om{\Omega}

\def\ze{{\zeta}}

\newcommand{\ov}{\overline}

\newcommand{\dist}{\operatorname{dist}}

\newcommand{\cont}{\operatorname{cont}}

\newcommand{\ID}{\mathop{\hbox{{\rm Id}}}\nolimits}

\newcommand{\I}{{\mathrm i}}
\newcommand{\dd}{{\mathrm d}}

\newcommand{\ee}{\mathrm e}

\newcommand{\ii}{^{-1}}

\newcommand{\ti}{\widetilde}

\newcommand{\ie}{{\emph{i.e.}}\ }

\newcommand{\eg}{{\it e.g.}\ }

\newcommand{\wrt}{{with respect to}}
\newcommand{\lhs}{{left-hand side}}
\newcommand{\rhs}{{right-hand side}}

\newcommand{\C}{\mathbb{C}}

\newcommand{\D}{\mathbb{D}}

\newcommand{\R}{\mathbb{R}}

\newcommand{\Z}{\mathbb{Z}}

\def\cB{\mathcal{B}}

\DeclarePairedDelimiter\abs{\lvert}{\rvert}%

\def\om{\omega}

\newcommand\Bbibitem[2]{\bibitem[#1]{#2}}

\newcommand{\defeq}{\coloneqq} 

\newcommand{\col}{\colon\thinspace}          

\newcommand{\begla}{\begin{equation}}
\newcommand{\beglab}[1]{\begin{equation}	\label{#1}}
\newcommand{\edla}{\end{equation}}

\newcommand{\imp}{\ens\Longrightarrow\ens}
\newcommand{\Imp}{\quad\Longrightarrow\quad}

\newcommand{\Li}{\operatorname{Li}}

\newcommand{\footremember}[2]{%
\footnote{#2}
\newcounter{#1}
\setcounter{#1}{\value{footnote}}%
}
\newcommand{\footrecall}[1]{%
\footnotemark[\value{#1}]%
}

\title{Hadamard Product and Resurgence Theory}

\author{%
  Y.~Li\footnote{Chern Institute of Mathematics 
  and LPMC, Nankai university, Tianjin 300071, China}
  \and
  D.~Sauzin\footnote{CNRS -- 
    Observatoire de Paris, PSL Research University, 75014 Paris, France}
  \footremember{CNU}{Department of Mathematics, Capital
    Normal University, Beijing 100048, China}
  \and
  S.~Sun\footrecall{CNU} \footnote{Academy for Multidisciplinary
    Studies, Capital Normal University, Beijing 100048, China}%
}

\date{December 2020}

\begin{document}

\thispagestyle{empty}

\maketitle

%








\begin{abstract}
  We discuss the analytic continuation of the Hadamard product of two
  holomorphic functions under assumptions pertaining to \'Ecalle's
  Resurgence Theory, proving that if both factors are endlessly
  continuable with prescribed sets of singular points~$A$ and~$B$, then so is
  their Hadamard product \wrt\ the set $\{0\}\cup A\cdot B$.
  {\bbl This is a generalization of the classical Hadamard Theorem in
    which all the branches of the multivalued analytic continuation of
    the Hadamard product are considered.}
  \end{abstract}






\bigskip



\section{Introduction}


\parage
The Hadamard product of two power series
$f(\xi), g(\xi) \in \C[[\xi]]$
is defined by the formula
\begla
f(\xi) = \sum_{n\ge0} a_n \xi^n
\quad\text{and}\quad
g(\xi) = \sum_{n\ge0} b_n \xi^n
\Imp
f\odot g(\xi) \defeq 
\sum_{n\ge0} a_n b_n \xi^n.
\edla
If the radii of convergence of~$f$ and~$g$ are positive, \ie
$f(\xi),g(\xi)\in\C\{\xi\}$, then
so is the radius of convergence of $f\odot g$.
From now on, we identify a convergent
power series with the holomorphic germ at~$0$ that it defines.
J.~Hadamard published in 1898 the proof of his now classical theorem \cite{Ha99},
to the effect that:
\begin{quote}
  \emph{Given finite subsets $A,B\subset\C^*$ and $K,L>0$, if $f$
    extends analytically to $\D_K \setminus \bigcup\limits_{\al\in A} R_\al$
  and $g$ to $\D_L \setminus \bigcup\limits_{\be\in B} R_\be$,
  then $f\odot g$ extends analytically to $\D_{KL}\setminus \bigcup\limits_{\ga\in
    A\cdot B} R_\ga$.}
  \end{quote}
  Here, we have denoted by $\D_M$ the open disc centred at the origin of
  radius~$M$,
  by~$R_\om$ the portion of the ray~$\om\,\R_{\ge0}$
  from~$\om$ onward,
  and by $A\cdot B$ the ``product set'' of~$A$ and~$B$:
  \begla
  \D_M \defeq D(0,M), \qquad
  R_\om \defeq \om\, [1,+\infty), \qquad
  A\cdot B \defeq \{ \al\be \mid (\al,\be)\in A\times B \}.
  \edla
  So, this is a statement about the ``principal branches'' of $f$, $g$
  and~$f\odot g$:
  we are considering their analytic continuation to a part of their
  Mittag-Leffler stars, assuming that any singular point that appears
  for~$f$ or~$g$ when following a ray emanating from the origin is
  isolated
  (in fact, J.~Hadamard also proves a more general version, in
  which the various $R_\om$'s are interpreted as portion of logarithmic spirals
  of the same polar slope).

In \cite{Bo98}, \'E.~Borel studied the analytic continuation of $f\odot
g$, also giving a proof of Hadamard's theorem, with a view to giving
more information on the nature of the singularities of~$f\odot g$ in
terms of the nature of the singularities of~$f$ and~$g$, at least in
the case of ``uniform singularities'' (\ie with trivial monodromy:
poles or essential singularities).
We refer the reader to \cite{PM20a}, \cite{PM20b} for a modern study and a
generalization of this result.

Interestingly, Borel also mentions something about the analytic
continuation of $f$, $g$ and $f\odot g$ to other sheets than the
principal one,\footnote{%
  In this article, when we speak of ``principal'' sheet or ``principal'' branch, we
  mean that we refer to the analytic continuation along straight line
  segments starting at the origin when possible, \ie we may identify the principal
  sheet with the Mittag-Leffler star.
}
and he gives a valuable hint in a footnote which has
been successfully used in the context of \'Ecalle's
Resurgence Theory to study the convolution product $f*g$---see \,\S\ref{quotationBorelftn}.4.
It is our aim to consider other branches than the principal one for
the analytic continuation of $f\odot g$, under appropriate assumptions
on~$f$ and~$g$.

\parage
In this article, we are concerned with the case when~$f$ and~$g$
satisfy a property pertaining to \'Ecalle's Resurgence Theory
\cite{Eca81}, namely a restricted form of ``endless continuability'':
\begin{Def}   \label{DefOmcont}
Given a non-empty closed discrete subset~$\Om$ of~$\C$, what we call
``$\Om$-continuable germ'' is a holomorphic germ~$f$ at the origin that admits
analytic continuation along any path $\ga\col [0,1] \to
\C\setminus\Om$ that has its initial point~$\ga(0)$ in the disc of convergence
of~$f$.
  \end{Def}
  This is the definition that has been employed in \cite{Eca81} with
  $\Om=2\pi\I\Z$, or in \cite{Sau13}, \cite{Sau15}, \cite{MS16}.
  It already covers a range of interesting examples, \eg in complex
  dynamics.
  The simplest examples of $\Om$-continuable germs are provided by
  meromorphic functions and algebraic functions; however, the
  definition puts no restriction on the nature of the singularities of
  the analytic continuation of an $\Om$-continuable germ, only on
  their location.
  More general is the definition of ``endless continuability''
  \cite{CNP93}, \cite{KS20}, and even more general that of
  ``continuability without a cut'' \cite{Eca85}.
  
  Our main result is
%
%
  \begin{thm}  \label{thmABcontinuable}
Given non-empty closed discrete subsets $A,B\subset\C$, if $f$ is an
$A$-continuable germ and~$g$ a $B$-continuable germ, then
\beglab{eqdefOmAB}
\Om \defeq \{0\} \cup (A\cdot B)
\edla
is closed and discrete and $f\odot g$ is an $\Om$-continuable germ.
\end{thm}
The proof is given in Sections~\ref{secdefid}--\ref{secvf}.


\smallskip

Note that in Definition~\ref{DefOmcont} we may have $0$ belonging
to~$\Om$ or not, but in all cases does an $\Om$-continuable germ have a
principal branch regular at~$0$. Simple examples with $\Om=\{0,1\}$
are
$f_1(\xi) \defeq -\frac{1}{\xi} \log(1-\xi)$ and
$\Li_2(\xi) \defeq \int_0^\xi f_1(\xi_1)\,\dd\xi_1$.
So, Theorem~\ref{thmABcontinuable} is compatible with the Hadamard
Theorem, and J.~Hadamard is right not to include~$0$ among the possibly
singular points of $f\odot g$
as far as only the principal branches are concerned.

The necessity of including~$0$ among the possibly singular points of
$f\odot g$ when we consider its analytic continuation on other sheets
was already noted by Borel, who gives credit to E.~Lindel\"of for that point
\cite[ftn~2]{Bo98}.
A simple example that illustrates it is
$f_0(\xi) \defeq -\log(1-\xi)$:
the germ $f_0$ is $\{1\}$-continuable, but the germ $f_0\odot f_0$ is
not; in fact, $f_0\odot f_0 = \Li_2$ is $\{0,1\}$-continuable.

  
\parage
The formal Borel transform $\cB\col t\,\C[[t]] \to \C[[\xi]]$ is defined
by:
\begla
\ti f(t) = \sum_{n\ge0} c_n t^{n+1}
\Imp
\cB\ti f(\xi) = \sum_{n\ge0} \frac{c_n}{n!} \xi^{n}.
\edla
In other words, $\cB\ti f(\xi) = \ee^\xi \odot \frac{\ti
  f(\xi)}{\xi}$.
A formal power series $\ti f(t) \in t\,\C[[t]]$ is said to be an
\emph{$\Om$-resurgent series} when its Borel transform $\cB\ti f$ is an
$\Om$-continuable germ \cite{Eca81}, \cite{Sau13}.
The singularities of~$\cB\ti f$ then give us information on the
divergence of~$\ti f(t)$, inasmuch as the Borel transform of a
convergent power series has no singularity at all.

Now,
\begla
f=\cB\ti f \quad\text{and}\quad g = \cB\,\ti g \Imp
f\odot g = \cB(e_1 \odot \ti f \odot\ti g),
\edla
where $e_1(t) \defeq t \, \ee^{t}$, hence our result can be rephrased as
\begin{quote}
  \emph{%
    Given non-empty closed discrete subsets $A,B\subset\C$, if $\ti f(t)$ is an
$A$-resurgent series and~$\ti g(t)$ a $B$-resurgent series, then
$e_1\odot \ti f(t)\odot \ti g(t)$ is an $\Om$-resurgent
series with~$\Om$ as in~\eqref{eqdefOmAB}.  }
  \end{quote}

  This is to be compared with Theorem~6.27 of \cite{MS16}, which can
  be rephrased as:
  \begin{quote}
    \emph{%
      Given non-empty closed discrete subsets $A,B\subset\C$ such that
    $\Om \defeq A \cup B \cup (A+B)$ is closed and discrete, if $\ti f(t)$ is an
$A$-resurgent series and~$\ti g(t)$ a $B$-resurgent series, then
their product $\ti f(t)\cdot \ti g(t)$ is an $\Om$-resurgent
series.  }
\end{quote}

The latter statement is a variation on a result on the multiplication
of $\Om$-resurgent series proved in \cite{Sau13}.
More generally, the stability under multiplication of the space of
resurgent series (defined as Borel preimages of endlessly continuable
germs or of germs continuable without a cut) is a key fact of
\'Ecalle's Resurgence Theory.
  
\parage
The shift of indices in the definition of~$\cB$ (mapping $t^{n+1}$ to
$\xi^n/n!$) is a convenient normalization which results in the
following relation between the ordinary multiplication in $\C[[t]]$
and the convolution~$*$ in $\C[[\xi]]$:
\beglab{eqaddconv}
\cB(\ti f \cdot \ti g) = f * g(\xi) \defeq \int_0^\xi f(\xi_1) g(\xi-\xi_1)\,\dd\xi_1.
\edla
This is the usual ``additive'' convolution (associated with Laplace
transform and Borel-Laplace summation), whereas the Hadamard product
may be considered as a ``multiplicative'' convolution. Indeed, as was
well-known already in Hadamard's time,
\beglab{eqmultconv}
0<\rho<R_f
\ens\text{and}\ens
\abs\xi < \rho R_g \Imp
f\odot g(\xi) = \frac{1}{2\pi\I}\oint_{C_\rho} f(\ze)
g\Big(\frac{\xi}{\ze}\Big) \frac{\dd\ze}{\ze},
\edla
where~$C_\rho$ is the parametrized circle: $s \in[0,1] \mapsto
\ze = \rho\,\ee^{2\pi\I s}$
(we have denoted by~$R_f$ and~$R_g$ the radii of convergence of~$f$ and~$g$).
\medskip

Actually, Formula~\eqref{eqmultconv} is the basis of Hadamard's and Borel's arguments in
\cite{Ha99} and \cite{Bo98}, and it will be the starting point of our
analysis as well.
The idea, translated from Borel's own words, is that ``this expression
of $f\odot g$ stays valid \emph{if one deforms the integration contour without
letting it cross any singular point of the integrand}'' and, ``the
contour having been fixed in an arbitrary manner, one obtains the
analytic continuation of $f\odot g$ \emph{by moving~$\xi$ in the
  plane, provided the singular points of the integrand do not cross
  the integration contour''}.

It seems that, throughout \cite{Bo98}, Borel keeps in mind the possibility of
going to the non-principal sheets and dealing with multivalued
analytic continuation, yet reluctantly so, since when he explicitly
mentions that possibility (in the aforementioned footnote~2 or
later in Sec.~III when he writes ``if the functions~$f$ and~$g$ were
multivalued...'') he tends to recommend to discard it (see the last
paragraph of his Sec.~III: ``It seems to us useless to insist on the
latter point... one would be led to complicated statements...'').
However, {\bbl he puts in a footnote (footnote~3) an idea that} has been successfully
adapted to study the analytic continuation of the \emph{additive}
convolution~\eqref{eqaddconv} of endlessly continuable germs in the context of
Resurgence \cite{Eca81}, \cite{CNP93}.
\label{quotationBorelftn}

We reproduce here the content of footnote~3 of \cite{Bo98}, with minor notational
changes. Borel writes~\eqref{eqmultconv} as
$f\odot g(\xi) = \frac{1}{2\pi\I}\oint_{C}
f\big(\frac{1}{x}\big) g(\xi x)
 \frac{\dd x}{x}$,
 with $C = C_{\rho\ii}$.

\begin{quotation}
``Let us conceive the closed contour $C$ as a flexible extensible
thread, the singular points of $f\big(\frac{1}{x}\big)$ as pins stuck
into the plane, the singular points of $g(\xi x)$ as pins that travel as
$\xi$ moves.
It is necessary and sufficient that the thread always part the two
systems of pins. Now, this will always be possible, by means of a
suitable deformation, if, while travelling, the second pins never come
to \emph{hit} the first ones (...); the thread may acquire a very
complicated form, but this is harmless." \cite[ftn~3, p.~240]{Bo98}
\end{quotation}

In the case of the additive convolution product~\eqref{eqaddconv},
instead of the closed contour $C=C_{\rho\ii}$, it is the line segment
$[0,\xi]$ that must be deformed.
In \cite{Sau13} {\bbl and \cite{MS16}}, it was shown how to contruct explicitly this
deformation
{\bbl when $\xi$ moves in the complex plane without meeting $A\cup B \cup (A+B)$};
{\bbl this is done} by applying to the initial integration contour a
homeomorphism that is a deformation of the identity obtained as the flow
of an explicit non-autonomous vector field tailored to the situation.

A benefit of such a detailed rigorous proof \wrt\ the arguments given
in \cite{Eca81} or \cite{CNP93} is that it allows for quantitative
estimates. 
{\bbl These estimates, in turn, can be}
used to show the stability of the space of resurgent series under
nonlinear operations and not only {\bbl multiplication} 
\cite{Sau15}, \cite{MS16}, \cite{KS20}.

\parage
Our proof of Theorem~\ref{thmABcontinuable} is spread over
\begin{enumerate}[--]
\item Section~\ref{secdefid}, which gives details on the kind of
``deformation of the identity'' that we need to deform the integration
contour in~\eqref{eqmultconv},
\item Section~\ref{secvf}, which explains how to obtain this
  deformation of the identity out of the
  flow of an explicit non-autonomous vector field.
  \end{enumerate}

This proof bears some resemblance with the proof of the theorem on
convolution presented in \cite{Sau13}, however a number of
modifications were necessary, notably a novel method to construct a
non-autonomous vector field adapted to Formula~\eqref{eqmultconv} (Section~\ref{secvf}).

\medskip

{\bbl We were led to Theorem~\ref{thmABcontinuable} while studying the Borel
transform of the Moyal star product of two resurgent series $\ti
f(t,q,p)$ and $\ti g(t,q,p)$
\cite{LSS20}. Indeed, the Borel counterpart of the Moyal star
product can be written in terms of the Borel transforms of the
factors, $f(\xi,q,p)$ and $g(\xi,q,p)$, and the formula appears as a
mixture of additive convolution \wrt~$\xi$ and Hadamard product; more
specifically, it involves the Hadamard product \wrt~$\ze$ of the
Taylor expansions 
$f(\xi_1,q,p+\ze)\odot g(\xi_2,q+\ze,p)$
and then a convolution-like integration \wrt~$\xi_1$ and~$\xi_2$.
However, in such a many-variable context, the restrictions one needs
to put on the singular locus of $f$ and~$g$ are more stringent than in
Definition~\ref{DefOmcont}, namely one requires
``algebro-resurgence'' \cite{GGS14}, \cite{LSS20}: $f,g\in\C\{\xi,q,p\}$ have analytic
continuation away from a proper algebraic subvariety of~$\C^3$
(or~$\C^{2N+1}$ if one is interested in deformation quantization of
$\C^{2N}$ with coordinates $q_1,\ldots,q_N,p_1,\ldots,p_N$).
As a result, the functions of~$\ze$ involved in the Hadamard product
part of the formula have at most finitely many singularities.}

\medskip

{\bbl 
Our work appears as complementary to the one of \cite{PM20a} and
\cite{PM20b}, which is devoted to the nature of the singularities of
the principal branch of $f\odot g$ and contains new formulas for the
monodromy at a given point $\om = \al \be$.
Considering the monodromy operator~$\De_\om$ forces to go around~$\om$ and visit other
sheets of the Riemann surface of $f\odot g$, but only the nearby
sheets (the ``next-to-principal'' one for $\De_\om(f\odot g)$, or the
ones reached by circling several times around~$\om$ if one wants to
iterate~$\De_\om$).
Dealing with these nearby branches of the analytic continuation of the
Hadamard product requires a
deformation of the integration contour in~\eqref{eqmultconv} that is
tractable by means of relatively elementary geometry;
the deformation becomes much more intricate if the
variable~$\xi$ is allowed to travel to arbitrary sheets, but it can still be
mathematically described and studied by following the method of
Sections~\ref{secdefid}--\ref{secvf}.
Therefore, it would be interesting to see whether this can be used to study the nature of
the singularities of the non-principal branches of $f\odot g$ in the spirit of
\cite{PM20a} and \cite{PM20b}.
Moreover, as noted in \cite{PM20a}, one might try to explore more deeply
the connection with \'Ecalle's Resurgence Theory; the so-called
``alien operators'', designed to study the singularities of endlessly
continuable functions, behave particularly well in relation with the
additive convolution~\eqref{eqaddconv} and they lead to contour
deformations reminiscent of those employed in \cite{PM20a}---see \eg
pp.~236--237 of \cite{MS16}.
}

\medskip

{\bbl
  Another avenue of research that naturally suggests itself would be
  to consider the Hadamard product of endlessly continuable germs~$f$ and~$g$
  which are not $\Om$-continuable in the sense of
  Definition~\ref{DefOmcont} for any discrete closed set~$\Om$.
  One may conjecture that $f\odot g$ is still endlessly continuable
  because, during the process of analytic continuation by means of
  contour deformation, one ``feels'' the presence of only finitely
  many singularities of~$f$ and~$g$ at the same time, which is why
  (additive) convolution (and even iterated convolutions) could be handled in \cite{KS20}; one
  might try to apply similar techniques to the Hadamard product.
}


\section{Deformations of the identity adapted to a path~$\ga$}   \label{secdefid}


We now begin the proof of Theorem~\ref{thmABcontinuable}.
We thus give ourselves $A$ and~$B$ non-empty closed discrete subsets
of~$\C$, and we define~$\Om$ by~\eqref{eqdefOmAB}.
We also introduce
\beglab{ineqab}
A' \defeq A \setminus\{0\}, \quad
a \defeq \min\{\abs\al \mid \al\in A'\}, \qquad
B' \defeq B \setminus\{0\}, \quad
b \defeq \min\{\abs\be \mid \be\in B'\},
\edla
regardless of $0$ belonging to $A\cup B$ or not.
Note that
\beglab{equivOm}
\Om = \{0\}\cup (A'\cdot B').
\edla


\parage
The first claim to be proved is
\begin{lem}
The set~$\Om$ is closed and discrete.
  \end{lem}
  \begin{proof}
    Clearly, this is equivalent to saying that
    \beglab{equivcl}
    \text{For any $R\in \R_{>0}$,} \quad \Om \cap \ov\D_R \ens\text{is finite.}
    \edla
    Let $R\in\R_{>0}$.
    Any pair $(\al,\be)\in A'\times B'$ such that $\al\be \in \ov\D_R$
    must satisfy $\abs\al\ge a>0$, $\abs\be\ge b>0$ and
    $\abs{\al\be}\le R$.
    Hence, $\al\in\ov\D_{R/b}$ and $\be\in\ov\D_{R/a}$.
    But $A\cap\ov\D_{R/b}$ and $B\cap\ov\D_{R/a}$ are finite,
    by the rephrasing~\eqref{equivcl} of our assumption on~$A$
    and~$B$, hence there only finitely many such pairs $(\al,\be)$ and
    $\Om\cap\ov\D_R$ is finite.
    \end{proof}


\parage
Let $I\defeq [0,1]$. We now give ourselves an $A$-continuable
germ~$f$, a $B$-continuable germ~$g$, and a path
$\ga\col I \to \C\setminus\Om$ that has its initial point close enough
to the origin, {\bbl specifically}
\beglab{hypgaz}
0 < \abs{\ga(0)} < ab.
\edla
Let us choose $\rho>0$ such that
\beglab{ineqrho}
\frac{\abs{\ga(0)}}{b} < \rho < a.
\edla
Note that the radius of convergence~$R_f$ of~$f$ is at least~$a$, and
the radius of convergence~$R_g$ of~$g$ at least~$b$,
hence Formula~\eqref{eqmultconv} implies that 
\beglab{eqfodotgaz}
f\odot g(\xi) = \frac{1}{2\pi\I}\oint_{C_\rho} f(\ze)
g\Big(\frac{\xi}{\ze}\Big) \frac{\dd\ze}{\ze}
\quad\text{for}\ens \xi \;\text{close enough to}\; \ga(0).
\edla

To prove Theorem~\ref{thmABcontinuable}, we just need to prove that
$f\odot g$ admits analytic continuation along~$\ga$.
%
Without loss of generality, we may suppose that~$\ga$ is $C^1$ with
a derivative~$\ga'$ that is Lipschitz.


\parage
We now make precise the notion of deformation of the identity that we
have alluded to in the introduction:
\begin{Def}   \label{Defdeformidga}
  We call ``deformation of the identity adapted to~$\ga$'' any family
  $(\Psi_t)_{t\in I}$ of Lipschitz homeomorphisms of~$\C$ such that
{\bbl $\Psi_0 = \ID_\C$ and}
  %
    %
  \begin{align}
    \label{eqPsital}
    \al\in \{0\}\cup A' &\Imp \quad\ens
    \Psi_t(\al) = \al \qquad\text{for all}\ens t\in I,\\[1ex]
    \label{eqPsitgabe}
  \be\in B'\quad &\Imp \Psi_t\Big(\frac{\ga(0)}{\be}\Big) = \frac{\ga(t)}{\be}
    \quad\text{for all}\ens t\in I.
  \end{align}
  %
  %
\end{Def}

\begin{lem}
  Suppose we have found~$(\Psi_t)_{t\in I}$, a deformation of the
  identity adapted to~$\ga$, then the analytic continuation of
  $f\odot g$ along~$\ga$ exists and is given, at each $\ga(t)$, $t\in
  I$, by the formula
  \beglab{contfodotg}
\cont_{\ga\mid t} (f\odot g)(\xi) = \frac{1}{2\pi\I}\oint_{\Psi_t(C_\rho)} f(\ze)
g\Big(\frac{\xi}{\ze}\Big) \frac{\dd\ze}{\ze}
\quad\text{for}\ens \xi \;\text{close enough to}\; \ga(t).
\edla
\end{lem}
  
\begin{proof}
  Since each $\Psi_t$ is a homeomorphism of~$\C$,
  condition~\eqref{eqPsital} {\bbl entails $\Psi_t\ii(\al)=\al$ for
    $\al\in A$, whence}
  \beglab{condzeA}
  \ze\in\C\setminus A \Imp \Psi_t(\ze)\in\C\setminus A
  \quad \text{for all}\ens t\in I.
  \edla
  {\bbl We also have $\Psi_t\ii(0)=0$ and, by
    condition~\eqref{eqPsitgabe}, $\Psi_t\ii\big(\frac{\ga(t)}{\be}\big) = \frac{\ga(0)}{\be}$
    for all $\be\in B'$, whence}
  \beglab{condgazeB}
  \ze\neq0 \ens\text{and}\ens \frac{\ga(0)}{\ze} \in\C\setminus B'
  \Imp
  {\bbl \Psi_t(\ze)\neq0 \ens\text{and}\ens }
  \frac{\ga(t)}{\Psi_t(\ze)} \in\C\setminus B'
  \quad \text{for all}\ens t\in I
  \edla
  (indeed: 
  %
  $\frac{\ga(t)}{\Psi_t(\ze)} =\be\in B'$ would require $\ze =
  \Psi_t\ii\big(\frac{\ga(t)}{\be}\big) = \frac{\ga(0)}{\be}$).
  Therefore, since by~\eqref{ineqrho} all the points of~$C_\rho$ satisfy the requirements
  indicated in the \lhs s of~\eqref{condzeA}--\eqref{condgazeB}, the
  integral in the \rhs\ of~\eqref{contfodotg} is {\bbl a well-defined
    function $h_t(\xi)$ of~$\xi$, which}
  represents a holomorphic germ at~$\ga(t)$.
  
The Cauchy theorem entails that, for $t<t'$ close enough, the
functions~$h_t$ and~$h_{t'}$ coincide in a neighbourhood of $\ga_{\mid
[t,t']}$, hence $h_t$ is the analytic continuation along~$\ga$
at~$\ga(t)$ of~$h_0$, and $h_0 =f\odot g$ by~\eqref{eqfodotgaz}.
\end{proof}



\section{Construction of a non-autonomous vector field}   \label{secvf}


\parage
To conclude the proof of Theorem~\ref{thmABcontinuable}, it is thus
sufficient to find a deformation of the identity in the sense of
Definition~\ref{Defdeformidga}.  Our tool will be the flow map induced
by a non-autonomous vector field.
\begin{lem}
Suppose $X\col I\times\C \to \C$ is locally Lipschitz
(identifying~$\C$ and~$\R^2$) and satisfies
\beglab{ineqboundX}
\abs{X(t,\ze)} \le K \abs\ze
\edla
with some $K>0$. Then,
for each $t^* \in I$ and each initial condition $\ze_0\in\C$,
the non-autonomous vector field
\begla
\frac{\dd\ze}{\dd t} = X(t,\ze)
\edla
has a unique maximal solution $t \mapsto \Phi^{t^*,t}(\ze_0)$ such
that $\Phi^{t^*,t^*}(\ze_0)=\ze_0$.
This maximal solution is defined for all $t\in I$.

Moreover, the map $(t^*,t,\ze) \in I\times I\times \C \mapsto
\Phi^{t^*,t}(\ze)\in\C$ thus defined is locally Lipschitz, 
and $(\Phi^{t^*,t})_{t^*,t\in I}$ is a family of locally Lipschitz
homeomorphisms of~$\C$ satisfying
\begla
\Phi^{t,t^{**}} \circ \Phi^{t^*,t} = \Phi^{t^*,t^{**}},
\qquad
\Phi^{t^*,t^*} = \ID_\C.
\edla
\end{lem}


\begin{proof}
Apply the classical Cauchy-Lipschitz theorem about the existence and
uniqueness of solutions to systems of ordinary differential equations
{\bbl and their regular dependence upon the initial condition}.
The maximal solutions are defined for all times because 
{\bbl of the bound~\eqref{ineqboundX}.}
\end{proof}


\begin{cor}   \label{corcXPsit}
  Suppose that we have found $c\col I\times\C \to \C$ locally
  Lipschitz, bounded by $K\abs\ze$ with some $K>0$, such that
    \begin{align}
    \label{eqcal}
    \al\in \{0\}\cup A' &\Imp \quad\ens
    c(t,\al) = 0 \qquad\text{for all}\ens t\in I,\\[1ex]
    \label{eqcgabe}
  \be\in B'\quad &\Imp c\Big(t,\frac{\ga(t)}{\be}\Big) = \frac{1}{\be}
    \quad\text{for all}\ens t\in I.
  \end{align}
  Then the flow map $\Psi_t \defeq \Phi^{0,t}$ of the non-autonomous
  vector field
  \begla
  X(t,z) \defeq c(t,z)\, \ga'(t)
  \edla
  gives rise to a deformation of the identity adapted to~$\ga$.
\end{cor}

\begin{proof}
  We have $\Psi_t(\al)=\Phi^{0,t}(\al)=\al$ for all $\al\in A$
  by~\eqref{eqcal},
  and~\eqref{eqcgabe} shows that $t\mapsto \frac{\ga(t)}{\be}$ is a
  particular solution of~$X$ for each $\be\in B'$, hence
  $\frac{\ga(t)}{\be} =
  \Phi^{0,t}\big(\frac{\ga(0)}{\be}\big) =
  \Psi_t\big(\frac{\ga(0)}{\be}\big)$.  
\end{proof}


\parage   \label{paragfindc}
Thus, we are just left with the question of finding a function~$c$
meeting the requirements of Corollary~\ref{corcXPsit}.
Since $\Om$ is closed and $\ga(I)$ is compact, we can pick $\de,M>0$ so
that, for each $t\in I$,
\beglab{ineqchoicedeM}
\de \le \abs{\ga(t)} \le M,
\qquad
\abs{\ga(t) - \al\be} \ge \de \quad\text{for all}\ens
(\al,\be) \in A\times B.
\edla


\begin{lem}
  If $\eps>0$ is chosen small enough so that
  \beglab{ineqchoiceeps}
  \eps< \frac{a}{2}, \qquad
  \frac{a\de}{\eps} > M + \frac{2M}{a} \eps,
  \edla
  then
  \beglab{ineqdistgabAp}
  \be\in B' \Imp \dist\Big(\frac{\ga(t)}{\be},A'\Big) \ge\eps
    \quad\text{for all}\ens t\in I.
  \edla
\end{lem}

\begin{proof}
Suppose $\abs{\frac{\ga(t)}{\be}-\al} <\eps$ with $\al\in A'$. We
would have
\beglab{ineqadeeps}
\de \le \abs{\ga(t) - \al\be} < \eps \abs\be
\edla
by~(\ref{ineqchoicedeM}b).
We also have $\abs{\al\be}\ge a\abs\be$ by~\eqref{ineqab} and
$\abs{\ga(t)}\le M$ by~(\ref{ineqchoicedeM}a),
whence 
\beglab{ineqtr}
a\abs\be \le \eps\abs\be + M
\edla
by the triangle inquality.
Now, either $\abs\be\ge \frac{2M}{a}$, but then 
$a\abs\be \le \eps\abs\be + M \le (\eps+\frac{a}{2})\abs\be$
leads to a contradiction with~(\ref{ineqchoiceeps}a);
or $\abs\be < \frac{2M}{a}$, but then 
$\frac{a\de}{\eps} < a\abs\be$ by~\eqref{ineqadeeps},
which is $\le M + \abs\be\eps$ by~\eqref{ineqtr}, 
whence the inequality
$\frac{a\de}{\eps} < M + \frac{2M}{a}\eps$,
which contradicts~(\ref{ineqchoiceeps}b).
\end{proof}


Let us choose a Lipschitz function $\eta=\eta_{\eps,A'} \col \C \to
[0,1]$ such that
\beglab{eqchoiceeta}
\al\in A' \imp \eta(\al) = 0,
\qquad
\dist(\ze,A') \ge \eps \imp \eta(\ze) = 1.
\edla
For instance, we may pick a Lipschitz function $\chi_\eps \col
\R_{\ge0} \to [0,1]$ such that $\chi_\eps(0)=0$ and
\[
d\ge\eps \imp \chi_\eps(d)=1,
\]
\eg $\chi_\eps(d) \defeq {\bbl\min}\big\{\frac{d}{{\bbl\eps}},1\big\}$, and define
$\eta(\ze) \defeq \chi_\eps\big(\dist(\ze,A')\big)$.
\medskip

We conclude by checking that the function defined by
\begla
%
c(t,\ze) \defeq \eta(\ze)
\frac{\ze}{\ga(t)}
%
\edla
fulfils all the requirements of Corollary~\ref{corcXPsit}:
\begin{enumerate}[--]
\item The function~$c$ is locally Lipschitz and bounded by
  ${\bbl\frac{1}{\de}}\abs\ze$, by virtue of~(\ref{ineqchoicedeM}a), because
  $\abs\eta\le1$.
\item We have $c(t,0)=0$.
\item If $\al\in A'$, then $c(t,\al)=0$ because~(\ref{eqchoiceeta}a)
  says that $\eta(\al)=0$.
\item If $\be\in B'$, then~\eqref{ineqdistgabAp} says that
  $\dist\big(\frac{\ga(t)}{\be},A'\big) \ge\eps$, hence
  $c\big(t,\frac{\ga(t)}{\be}\big) =
  \eta\big(\frac{\ga(t)}{\be}\big)\frac{1}{\be} = \frac{1}{\be}$
  by~(\ref{eqchoiceeta}b).
\end{enumerate}
Therefore, the proof of Theorem~\ref{thmABcontinuable} is complete.

\medskip


\parage
\textbf{Remark.}\ 
When~$B$ is a finite set, the construction of \S\,\ref{paragfindc}.2
can be replaced by one closer to that of \cite{Sau13}, namely:
\begla
c(t,\ze) \defeq
\sum_{\be_1\in B'}\,\, \frac{ \eta_{\be_1}(\ze,t) }{ \eta_{\be_1}(\ze,t) +
  \abs{\ze-\frac{\ga(t)}{\be_1}} }
\cdot \frac{1}{\be_1},
\qquad\text{with}\quad
\eta_{\be_1}(\ze,t) \defeq \dist \big(\ze, \{0\}\cup A' \cup
B(\be_1,t) \big),
\edla
where
$B(\be_1,t) \defeq \big\{ \frac{\ga(t)}{\be_2} \mid \be_2 \in
B'\setminus\{\be_1\} \big\}$.
Indeed, the denominators are all always $>0$ precisely because
$\ga(I)$ and $\{0\}\cup (A'\cdot B')$ do not intersect.


\vspace{1cm}

\subsubsection*{Acknowledgements}

The second author thanks Capital Normal University for their hospitality
during the period September 2019--February 2020, where this research
was started.
The third author is partially supported by NSFC (No.s 11771303, 11911530092, 11871045).
\medskip

This paper is partly a result of the ERC-SyG project, Recursive and
Exact New Quantum Theory (ReNewQuantum) which received funding from
the European Research Council (ERC) under the European Union's Horizon
2020 research and innovation programme under grant agreement No
810573.


\end{document}